\documentclass[twoside]{article}
\usepackage[accepted]{aistats2018}
\usepackage{amsmath,amsthm,epsfig,amssymb,amsbsy}
\usepackage{mathtools}
\usepackage{mathrsfs}
\usepackage{relsize}
\usepackage{nicefrac} 
\usepackage{caption} 
\usepackage{subcaption} 

\usepackage{pgfplots}
\pgfplotsset{compat=newest}
\usetikzlibrary{plotmarks}
\usetikzlibrary{arrows.meta}
\usepgfplotslibrary{patchplots}
\usepackage{grffile}

\usepackage{booktabs}       % professional-quality tables

\newtheorem{algorithm}{\textbf{\textup{Algorithm}}}

\newtheorem{lem}{Lemma}
\newtheorem{thm}{Theorem}

\newtheorem{definition}{Definition}
\newtheorem{remark}{Remark}

\newcommand{\R}{\mathbb{R}}
\newcommand{\N}{\mathbb{N}}
\newcommand{\U}{\mathcal{U}}

\DeclareMathOperator*{\argmin}{arg\,min}

\DeclareMathOperator*{\hyp}{hypo}

\begin{document}

\runningtitle{A Nonconvex Proximal Splitting Algorithm under Moreau-Yosida Regularization}

\runningauthor{Laude, Wu, Cremers}

\twocolumn[
%
%\aistatstitle{A Primal-Dual Algorithm for Nonconvex Nonsmooth Problems \\ under Moreau-Yosida Regularization}
\aistatstitle{A Nonconvex Proximal Splitting Algorithm \\
under Moreau-Yosida Regularization}
%\aistatsauthor{ Author 1 \And Author 2 \And  Author 3 }
%
%\aistatsaddress{ Institution 1 \And  Institution 2 \And Institution 3 } 
\aistatsauthor{ Emanuel Laude \And  Tao Wu \And Daniel Cremers }
\aistatsaddress{Department of Informatics, Technical University of Munich, Germany} 
]

\begin{abstract}
We tackle highly nonconvex, nonsmooth composite optimization problems whose objectives comprise a Moreau-Yosida regularized term. Classical nonconvex proximal splitting algorithms, such as nonconvex ADMM, suffer from lack of convergence for such a problem class. To overcome this difficulty, in this work we consider a lifted variant of the Moreau-Yosida regularized model and propose a novel multiblock primal-dual algorithm that intrinsically stabilizes the dual block. We provide a complete convergence analysis of our algorithm and identify respective optimality qualifications under which stationarity of the original model is retrieved at convergence. Numerically, we demonstrate the relevance of Moreau-Yosida regularized models and the efficiency of our algorithm on robust regression as well as joint feature selection and semi-supervised learning.
\end{abstract}

\section{Introduction} \label{sec:intro}
%% original model
%Many relevant problems in statistics and machine learning take the form as follows:
%\begin{equation}\label{eq:consensus_prob}
%\begin{aligned}
%\underset{\substack{u\in \R^n, ~v \in \R^m}}{\text{minimize}}&& f (v) + g(u) \\
%\text{subject to}
%&& A u = v,
%\end{aligned}
%\end{equation}
We are interested in the solution of the nonconvex, Moreau-Yosida regularized consensus minimization problem
\begin{equation}\label{eq:consensus_prob_regularized}
\begin{aligned}
\underset{\substack{u\in \R^n, ~v \in \R^m}}{\text{minimize}}&& \textstyle e_{\lambda} f (v) + g(u) \\
\text{subject to}
&& A u = v,
\end{aligned}
\end{equation}
where $e_{\lambda} f$ is the \emph{Moreau envelope} \cite{moreau1965proximite} of $f$ with parameter $\lambda > 0$.

\begin{figure}[htb]
\begin{subfigure}[b]{0.47\linewidth}
        \centering
        \includegraphics[width=0.9\textwidth]{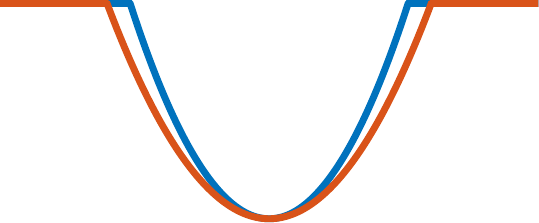}
%        \resizebox{\textwidth}{!}{
%	\input{./moreau_tq.tex}
%	}
\end{subfigure}
\hfill
\begin{subfigure}[b]{0.47\linewidth}
        \centering
        \includegraphics[width=0.9\textwidth]{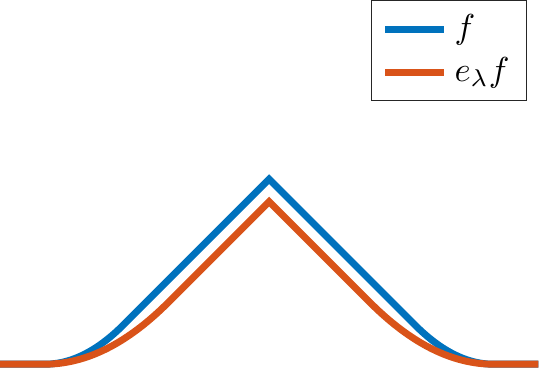}
%        \resizebox{\textwidth}{!}{
%	\input{./moreau_symm_huber_hinge.tex}
%	}
\end{subfigure}
\caption{Illustration of nonconvex, nonsmooth Moreau envelopes $e_\lambda f$ (red curves). The two depicted piecewise quadratic loss functions $f$ (blue curves), are invariant under Moreau-Yosida regularization up to rescaling. Left: The Moreau envelope of a truncated quadratic loss, given as $f=\min\;\{\nu, (\cdot)^2\}$ for $\nu > 0$, is a nonsmooth, nonconvex truncated quadratic loss with rescaled slope. Right: The Moreau envelope of a symmetric Huberized hinge loss is a nonsmooth, nonconvex symmetric Huberized hinge loss with rescaled slopes of the quadratic pieces given as $e_\lambda f =\min_{\theta\in \{-1,1\}} e_\lambda (1- (\cdot) \theta)_+$.}
\label{fig:moreau_envelope}
\end{figure}

The extended real-valued functions, $f$ and $g$, are assumed to be proper, lower semi-continuous (lsc), and in general {nonconvex}. 
%The objective is \emph{separable} in the block variables $u$ and $v$ that comply with a \emph{consensus constraint} with matrix $A \in \R^{m \times n}$.
The block variables $u$ and $v$ are separated in the objective, but comply with a \emph{consensus constraint} for matrix $A \in \R^{m \times n}$.
%Such a formulation is amenable to distributed solution methods such as (accelerated) proximal gradient methods \cite{nesterov2013gradient,ochs2014ipiano} and the alternating direction method of multipliers (ADMM) \cite{glowinski1975approximation,gabay1976dual,gabay1983chapter,eckstein1992douglas,hong2016convergence,li2015global}.
%Algorithmically, these proximal splitting algorithms typically update the block variables in a Gauss-Seidel fashion, and exploit the blessing that the \emph{proximal mapping} of $f$ and/or $g$ can be efficiently computed. 
%The convergence (or even applicability) is, however, only guaranteed under rather restrictive conditions: At least one of the functions between $f$ and $g$ is required to be Lipschitz differentiable, cf.~\cite{hong2016convergence,li2015global,HHZW16,WYZ17}. 
%In some sense, convexity is traded off against smoothness in a convergent algorithm. When both $f$ and $g$ are nonsmooth and nonconvex, the convergence of classical proximal splitting algorithms is lost.

To argue that model \eqref{eq:consensus_prob_regularized} is relevant in statistics and machine learning we briefly review some general properties of the Moreau envelope, formally defined as follows.
\begin{definition}[Moreau envelope and proximal mapping] For a proper, lsc function $f:\R^m \to \R \cup \{ \infty \}$ and parameter $\lambda > 0$, the Moreau envelope function $e_\lambda f$ and the proximal mapping $P_\lambda f$ are defined by
\begin{align*}
e_\lambda f(v) &:= \inf_{z\in\R^m} f(z) + \frac{1}{2\lambda} \|z-v\|^2, \\
P_\lambda f(v) &:= \argmin_{z\in\R^m} f(z) + \frac{1}{2\lambda} \|z-v\|^2.
\end{align*}
\end{definition}
For convex $f$, the Moreau envelope $e_\lambda f$ yields a convex and smooth lower approximation to $f$, cf.~\cite[Theorem 2.26]{Rockafellar-Variational-Analysis}. In contrast, for nonconvex nonsmooth $f$ the Moreau envelope $e_\lambda f$ remains nonsmooth and nonconvex in general which renders its optimization challenging.
%A well-known example is the Huber loss, which can be written in terms of the Moreau envelope $e_\lambda |\cdot|$ of the absolute value function $|\cdot|$.
%The Moreau envelope also has a geometrical interpretation. Its epigraph $\mathrm{epi}\,e_\lambda f := \{(\tb{z},q) \in \R^{m+1} : e_\lambda f(\tb{z}) \leq q\}$ is the Minkowski sum of the epigraph of $f$ and the epigraph of the quadratic function $\frac{1}{2\lambda} \|\cdot \|^2$ \cite[Exercise 1.28]{Rockafellar-Variational-Analysis}. In the limit case where $\lambda \to 0$, the Moreau envelope converges pointwisely to the original function $f$ \cite[Theorem 1.25]{Rockafellar-Variational-Analysis}.
In fact, many nonsmooth, nonconvex objectives in machine learning tasks, such as semi-supervised Huber-SVMs, or robust regression are even invariant, up to different choices of $\lambda$, under Moreau-Yosida regularization. This is exemplarily illustrated for the truncated quadratic loss \cite{she2011outlier,LiuMinimizingTQ}, and the ``symmetric Huberized hinge loss'' \cite{bennett1999semi,joachims1999transductive,collobert2006large} in Figure \ref{fig:moreau_envelope}, where Moreau-Yosida regularization preserves nonsmoothness and nonconvexity of $f$.

For $e_{\lambda} f$ smooth and $g$ nonsmooth, problems of the form \eqref{eq:consensus_prob_regularized} can be optimized by proximal splitting methods such as (accelerated) proximal gradient methods \cite{nesterov2013gradient,ochs2014ipiano}, or primal-dual algorithms such as the alternating direction method of multipliers (ADMM) \cite{glowinski1975approximation,gabay1976dual,gabay1983chapter,eckstein1992douglas,hong2016convergence,li2015global,WYZ17}.
All these algorithms typically update the block variables in a Gauss-Seidel fashion, and exploit efficient evaluation of the proximal mapping of $e_{\lambda} f$ or $g$. 
%However, the nonsmoothness of $e_{\lambda} f$ renders proximal gradient methods inapplicable and ADMM possibly non-convergent. 
Proximal splitting methods are traditionally applied for convex optimization, in which case their convergence is guaranteed without smoothness of the objective. 
More recently, convergence of ADMM is extended to certain nonconvex settings \cite{hong2016convergence,li2015global,WYZ17}, but the convergence analysis requires, e.g., that $e_{\lambda} f$ is Lipschitz differentiable. In some sense, convexity is traded off against smoothness in a convergent algorithm. However, when both $e_{\lambda} f$ and $g$ in \eqref{eq:consensus_prob_regularized} are nonsmooth and nonconvex (as encountered in numerous machine learning applications), 
the proximal gradient method is inapplicable and the ADMM is non-convergent.

%Yet, ADMM is often applied as a heuristic method, along with adaptive step size schemes approaching $0$ \cite{}. This may enforce convergence of ADMM in some cases, however, it cannot be guaranteed that the iterates approach a critical point of problem \eqref{eq:consensus_prob_regularized}.

%% our contribution
To overcome such difficulties, we propose a novel multiblock primal-dual scheme for solving problem \eqref{eq:consensus_prob_regularized}. 
%that converges (\tb{subsequentially}) to a critical point of a \emph{lifted} representation of the nonsmooth nonconvex problem \eqref{eq:consensus_prob_regularized} that is equivalent to the original problem in terms of global minimizers.
Our contributions are summarized as follows:
%In summary, this work makes the following contributions:
\begin{itemize}
\item We devise a novel multiblock primal-dual scheme in Section \ref{sec:derivation}, and prove its subsequential convergence to a critical point of the \emph{lifted} representation of problem \eqref{eq:consensus_prob_regularized} in Section \ref{sec:convergence}.
%\tb{The introduction of a new block variable due to Moreau-Yosida regularization intrinsically regularizes the behavior of the dual variable.}
\item We draw connections of our primal-dual algorithm to a fully primal coordinate descent on a quadratic penalty function; see Section \ref{sec:specialcase}. %In fact, we show that for a certain choice of step size our method becomes fully primal.
%\item We prove in Section \ref{sec:convergence} that for Lipschitz $f$ and $\lambda \to 0^+$, we consistently solve the unregularized problem, where $e_\lambda f$ is replaced by $f$, to stationarity. The violation of the linear constraint in \eqref{eq:consensus_prob} is quantified in terms of $\lambda$.
\item For piecewise smooth, piecewise convex functions of the form $\min_{i \in \mathcal{I}} e_\lambda f_i$ (as a subfamily of nonconvex, nonsmooth functions), 
we identify an optimality qualification related to \emph{active set}, under which a critical point of the lifted problem translates to a critical point of \eqref{eq:consensus_prob_regularized}; see Section \ref{sec:piecewiseconvex}. For a specific choice of the step size, we show that our optimality qualification is guaranteed to hold.
\item We experimentally validate our proposed algorithm on robust regression as well as joint feature selection and semi-supervised learning in Section \ref{sec:numerics}. In comparison with classical nonconvex ADMM, our method consistently performs favorably in terms of lower objective value and vanishing optimality gap.
\end{itemize}

\section{A Multiblock Primal-Dual Method} \label{sec:algorithm}
\subsection{Derivation} \label{sec:derivation}
For $e_\lambda f$ smooth, the convergence of ADMM in the nonconvex setting is shown via a monotonic decrease of the augmented Lagrangian \cite{hong2016convergence}, which serves as a Lyapunov function. As a key step in the proof of \cite{hong2016convergence}, the ascent of the augmented Lagrangian, caused by the dual update, is dominated by a sufficient descent in the primal block that is updated last. 

In order to recover convergence for nonsmooth $e_\lambda f$, we employ a lifting of the problem which yields a third primal block in the optimization. More precisely, we introduce new variables $z, w$ along with a linear constraint 
\begin{align} \label{eq:substitution}
z + w = v,
\end{align}
 and integrate the Moreau-Yosida regularization into the lifted problem:
\begin{equation}\label{eq:consensus_prob_lifted}
\begin{aligned}
\underset{\substack{u\in \R^n, ~w,z \in \R^m}}{\text{minimize}}& & f(z) + \frac{1}{2\lambda}\|w\|^2 + g(u) \\
\text{subject to}
& & A u - z - w = 0.
\end{aligned}
\end{equation}
To this lifted problem we apply the following multiblock primal-dual scheme, where the block $w$, updated last, corresponds to the smooth function $\frac{1}{2\lambda}\|\cdot\|^2$ that realizes the Moreau-Yosida regularization.
Let the augmented Lagrangian of the problem be defined as
\begin{equation} \label{eq:augm_lagrangian}
\begin{aligned} 
&\mathfrak{L}_{\rho}(u,z,w,y)=f(z) + g(u)  +\frac{1}{2\lambda}\|w\|^2 \\
&\qquad+ \langle y, Au - z -w \rangle + \frac{\rho}{2} \|A u - z - w\|^2,
\end{aligned}
\end{equation}
for $\rho > 0$. Also let $M:=\frac{1}{\sigma}I - \rho A^\top A$, which is positive definite for $\sigma\rho\|A\|^2<1$, and $\|\cdot\|_M:=\sqrt{\langle\cdot,M\cdot\rangle}$. Then our scheme is formulated as
\begin{equation}\label{eq:primal_dual}
\begin{aligned} 
u^{t+1} &\in \textstyle\argmin_{u} \mathfrak{L}_{\rho}(u,z^t,w^t,y^t) + \frac{1}{2}\|u-u^t\|_M^2, \\
z^{t+1}&\in \textstyle\argmin_{z} \mathfrak{L}_{\rho}(u^{t+1},z,w^{t},y^t), \\
w^{t+1}&=\textstyle\argmin_{w} \mathfrak{L}_{\rho}(u^{t+1},z^{t+1},w,y^t), \\
y^{t+1}&=y^t + \rho(A u^{t+1} - z^{t+1} - w^{t+1}).
\end{aligned}
\end{equation}
Motivated by \cite{chambolle2011first,lin2011linearized} we include a proximal term $\frac{1}{2}\|u-u^t\|_M^2$ to guarantee tractability of the subproblem for the $u$-update, as long as the proximal mapping of $g$ is simple.
Once rephrasing the update of $u$ in terms of a proximal mapping as $u^{t+1} = \textstyle P_{\sigma}g(u^t - \sigma A^\top(y^{t} + \rho  (A u^{t} -z^{t} - \lambda y^t)))$, one can interpret it as a proximal gradient descent step on the augmented Lagrangian.

The optimality condition for the last primal block update $0 \stackrel{!}{=} \frac{1}{\lambda} w^{t+1} - y^t - \rho(A u^{t+1} - z^{t+1} - w^{t+1})$ matches the dual update and shows that $w$ is equal to the dual variable up to scaling, i.e.,~$\lambda y^{t+1}=w^{t+1}$. Therefore, the variable $w$ can be eliminated from the algorithm, and we arrive at a compact formulation of the proposed multiblock primal-dual scheme, i.e.,~Algorithm \ref{alg:multiblock_primal_dual}, for Moreau-Yosida regularized problems.
\begin{figure}[h]
\centering
\fbox{
\begin{minipage}{0.9\linewidth}
\begin{algorithm}[multiblock primal-dual scheme] \label{alg:multiblock_primal_dual}
{\ }
Choose $\rho, \sigma$ such that $\rho\lambda > 1$ and $\sigma\rho\|A\|^2 < 1$.
For $t=1, 2, \dots$ do
\begin{align*}
u^{t+1} &\in \textstyle P_{\sigma}g\left(u^t - \sigma A^\top(y^{t} + \rho  (A u^{t} -z^{t} - \lambda y^t))\right),\\
        z^{t+1} &\in \textstyle P_{\nicefrac{1}{\rho}}f\left(A u^{t+1} + ( \nicefrac{1}{\rho}-\lambda) y^t \right), \\
	y^{t+1}&= \tfrac{1}{1 + \rho \lambda}(y^t + \rho(A u^{t+1} - z^{t+1})).
\end{align*}
\end{algorithm}
\end{minipage}
}
\end{figure}

Note that the lifted problem is equivalent to problem \eqref{eq:consensus_prob_regularized} in terms of global minimizers but in general not in terms of critical points. 
Yet, we show in Section \ref{sec:qualification} that under mild assumptions, e.g., for piecewise convex piecewise smooth $e_\lambda f$, limit points produced by this algorithm translate to critical points of the original problem \eqref{eq:consensus_prob_regularized} by reversing the substitution \eqref{eq:substitution}.
As a side remark, with $\lambda = 0$ we arrive at a proximal variant of ADMM \cite{chambolle2011first,lin2011linearized} also referred to as linearized ADMM, applied to the unregularized problem
\begin{equation}\label{eq:consensus_prob}
\begin{aligned}
\underset{\substack{u\in \R^n, ~v \in \R^m}}{\text{minimize}}&& f (v) + g(u) \\
\text{subject to}
&& A u = v.
\end{aligned}
\end{equation}

\subsection{Primal Optimization as a Special Case} \label{sec:specialcase}
Finally, we draw a connection between Algorithm \ref{alg:multiblock_primal_dual} and existing, fully primal alternating minimization schemes applied to the quadratic penalty
\begin{align} \label{eq:penalty_function}
Q(u, z) = f(z) +g(u)  + \frac{1}{2\lambda}\|A u - z\|^2,
\end{align}
that corresponds to the unregularized problem \eqref{eq:consensus_prob}.
For a non-admissible choice of the step size $\nicefrac{1}{\rho}=\lambda$, the Lagrange multiplier $y^t$ in Algorithm \ref{alg:multiblock_primal_dual} becomes obsolete and we arrive at Algorithm \ref{alg:proximal_penalty_method}, a fully primal Gauss-Seidel minimization of \eqref{eq:penalty_function} over $u,z$.
\begin{figure}[h]
\centering
\fbox{
\begin{minipage}{0.9\linewidth}
\begin{algorithm}[proximal penalty method] \label{alg:proximal_penalty_method}
{\ }
Choose $\tau,\sigma>0$ such that $\sigma\|A\|^2 < \lambda$.
For $t=1, 2, \dots$ do
\begin{align*}
u^{t+1} &\in \textstyle P_{\sigma}g\left(u^t - \nicefrac{\sigma}{\lambda}  A^\top(A u^{t} -z^{t}))\right),\\
     z^{t+1} &\in \textstyle P_{\lambda}f\left(A u^{t+1} \right).
\end{align*}
\end{algorithm}
\end{minipage}
}
\end{figure}
The above algorithm appears similar in form to \emph{proximal alternating linearized minimization} (PALM) \cite{bolte2014proximal} applied to \eqref{eq:penalty_function}, where $H(u,z):=\frac{1}{2\lambda}\|A u - z\|^2$ is interpreted as the differentiable coupling term. To update $z$, PALM invokes a second proximal gradient descent step on \eqref{eq:penalty_function} with step size $\tau < \lambda$, i.e.,~$z^{t+1}=\textstyle P_{\tau}f(z^t + \frac{\tau}{\lambda} (A u^{t+1} - z^{t}))$. With the non-admissible step size $\tau=\lambda$, we recast Algorithm \ref{alg:proximal_penalty_method} from PALM.

\section{Convergence Analysis} \label{sec:convergence}
Our convergence proof borrows arguments from \cite{hong2016convergence}, where the convergence of ADMM was shown via a monotonic decrease of the augmented Lagrangian. In our case a Lyapunov function that monotonically decreases over the iterations is obtained by eliminating the variable $w$ from the augmented Lagrangian \eqref{eq:augm_lagrangian}:
\begin{equation} \label{eq:merit}
  \begin{aligned}
&\mathfrak{Q}_{\rho}(u, z, y) = f(z) - \frac{\lambda}{2}\|y\|^2 +g(u) \\
&\qquad+ \langle A u - z , y \rangle+ \frac{\rho}{2} \| A u - z- \lambda y\|^2.
  \end{aligned}
  %\vspace{-0.2cm}
\end{equation}
The following lemma is the central part of our convergence proof, as it guarantees convergence of $\mathfrak{Q}_{\rho}(u^{t}, z^{t}, y^{t})$. Our algorithm has three blocks and invokes a proximal gradient descent step on $\mathfrak{Q}_{\rho}(\cdot, z^{t}, y^{t})$ to update $u$. This keeps all block variable updates computationally tractable, as long as the proximal mappings of $f$ and $g$ are simple. 
\begin{lem} \label{lem:boundedness_suff_desc}
Let $\lambda > 0$. Choose $\rho>0$ sufficiently large so that $\lambda \rho \geq 1$, and then $\sigma>0$ sufficiently small so that $\sigma \rho \|A\|^2 < 1$. Assume that $e_{\lambda} f(A\, \cdot) + g(\cdot)$ is bounded from below. 
Then the following statements hold true:
\begin{enumerate}
\item
$\mathfrak{Q}_{\rho}(u^{t+1}, z^{t+1}, y^{t+1})$ is an upper bound of the quadratic penalty \eqref{eq:penalty_function} at $(u^{t+1}, z^{t+1})$, i.e.,
\begin{equation*}
\begin{aligned}
Q(u^{t+1}, z^{t+1}) \leq \mathfrak{Q}_{\rho}(u^{t+1}, z^{t+1}, y^{t+1}).
\end{aligned}
\end{equation*}
\item
The sequence $\{\mathfrak{Q}_{\rho}(u^{t}, z^{t}, y^{t})\}_{t\in \N}$ is bounded from below.
\item
A sufficient decrease over $\mathfrak{Q}_{\rho}$ for each iteration is guaranteed:
\begin{equation*}
%\label{eq:suff_descent}
\begin{aligned}
  & \mathfrak{Q}_{\rho}(u^{t+1}, z^{t+1}, y^{t+1})- \mathfrak{Q}_{\rho}(u^{t}, z^{t}, y^{t}) \\
  & \quad\leq \left( \tfrac{\rho \|A\|^2}{2} - \tfrac{1}{2 \sigma} \right)\|u^{t+1}-u^t\|^2 \\ 
  &\quad \quad +\left(\tfrac{1}{\rho}-\tfrac{\rho\lambda^2 + \lambda}{2} \right) \|y^{t+1}-y^{t}\|^2.
\end{aligned}
\end{equation*}
\end{enumerate}
\end{lem}

Our aim is to prove subsequential convergence of the iterates $\{(u^t, z^t, y^t)\}_{t\in \N}$ to a critical point of \eqref{eq:consensus_prob_lifted}. To guarantee the existence of such a subsequence, the iterates shall be bounded. To ensure that the conditions of a critical point are met at limit points, the distance of consecutive iterates shall vanish in the limit. All of these are established in the following lemma. 
\begin{lem} \label{lem:limit_feasibility}
Let the assumptions be as in Lemma \ref{lem:boundedness_suff_desc}, and let the quadratic penalty in \eqref{eq:penalty_function} be coercive. Then, the iterates $\{(u^t,z^t)\}_{t \in \N}$ are uniformly bounded and 
%the difference of two consecutive iterates $u^{t+1}, u^t$ goes to zero:
\begin{align*}
\|u^{t+1}-u^t\| \to 0,
\end{align*}
as $t\to\infty$.
If in particular $\rho\lambda > 1$, then the iterates $\left\{y^t\right\}_{t\in\N}$ are uniformly bounded, and feasibility is achieved in the limit, i.e.,
\begin{align*}
\|A u^{t} - z^t - \lambda y^t\| \to 0,
\end{align*}
and furthermore
\begin{align*}
&\|z^{t+1}-z^{t}\| \to 0, \\
&\|y^{t+1}-y^{t}\| \to 0.
\end{align*}
%as $t\to\infty$.
\end{lem}
We remark that the second part of the lemma holds for the case $\rho\lambda = 1$, i.e., it refers to the case when our Algorithm \ref{alg:multiblock_primal_dual} specializes to Algorithm \ref{alg:proximal_penalty_method}. 
%\begin{proof}
%Since by Lemma \ref{lem:suff_descent} $\mathfrak{Q}_{\rho}(u^{t}, z^{t}, y^{t})$ monotonically decreases it is upper bounded. Since $\mathfrak{Q}_{\rho}(u^{t}, z^{t}, y^{t})$ is an upper bound for the quadratic penalty $Q(u^t,z^t)$ and since $Q$ is coercive $u^t$ and $z^t$ are bounded.
%We sum the estimate form Lemma \ref{lem:suff_descent} from $t=1$ to $T$ and obtain due to the lower boundedness of the iterates $\mathfrak{Q}_{\rho}(u^{t}, z^{t}, y^{t})$ from Lemma \ref{lem:boundedness}
%\begin{align*}
%- \infty &< \mathfrak{Q}_{\rho}(u^{T+1}, z^{T+1}, y^{T+1})- \mathfrak{Q}_{\rho}(u^{1},z^{1}, y^1) \\
%&\leq -\sum_{t=1}^T \left( \tfrac{\rho \|A\|^2 }{2} - \tfrac{1}{2 \sigma} \right)\|u^{t+1}-u^t\|^2 \\
%  &\quad -\sum_{t=1}^T \left(\tfrac{1}{\rho} - \tfrac{\rho\lambda^2 + \lambda}{2}\right) \|y^{t+1}-y^{t}\|^2.  
%\end{align*}
%Letting $T \to \infty$ yields that $\|u^{t+1}-u^t\| \to 0$ and $\| y^{t+1} - y^t \| \to 0$ for sufficiently large $\rho$.
%From $\frac{1}{\rho}(y^{t+1}-y^t) =A u^{t+1} -z^{t+1} - \lambda y^{t+1}$ we have that, 
%\begin{align*}
%	0 &\leq \|z^t-z^{t+1}\| \\
%	&= \|-z^{t+1} + z^t + A(u^{t+1}-u^{t}) - A(u^{t+1}-u^{t})\\
%	&\qquad +\lambda y^{t+1} - \lambda y^t - \lambda y^{t+1} + \lambda y^t \| \\
%	&\leq \frac{1}{\rho} \| y^{t+1} - y^t \| + \|A\|\|u^{t+1}-u^{t}\| \\
%	&\qquad +\lambda \|y^{t+1} - y^t  \| \to 0
%\end{align*}
%and $\|A u^{t} - z^t - \lambda y^t\| \to 0$. This also shows that $y^t$ is bounded.
%\end{proof}

It remains to show that limit points of the algorithm correspond to critical points of the lifted problem \eqref{eq:consensus_prob_lifted}. To this end we define the subgradient 
%$\partial f$ 
\cite[Definition 8.3]{Rockafellar-Variational-Analysis} as follows.
\begin{definition}[subgradients] \label{def:general_subdiff}
Consider a function $f:\R^{m} \to \R \cup \{ \infty \}$ and a point $\bar{v}$
with $f(\bar{v})$ finite. For a vector $y\in \R^{m}$, one says that
\begin{enumerate}
\item $y$ is a regular subgradient of f at $\bar{v}$, written $y \in \hat{\partial} f(\bar{v})$, if
\begin{align*}
\liminf_{\substack{v \to \bar{v} \\ v \neq \bar{v}}} \frac{f(v)-f(\bar{v})- \langle y,v-\bar{v} \rangle }{\|v-\bar{v}\|} \geq 0.
\end{align*}
\item $y$ is a (general) subgradient of $f$ at $\bar{v}$, written $y \in \partial f(\bar{v})$, if there are sequences $v^t \to \bar{v}$ with $f(v^t) \to f(\bar{v})$ and $y^t \in \hat{\partial} f(v^t)$ with $y^t \to y$.
\end{enumerate}
\end{definition}
In the following theorem we guarantee subsequential convergence to critical points of \eqref{eq:consensus_prob_lifted}. Having proven that the difference of consecutive iterates vanishes in the limit, optimality follows directly from the optimality conditions of the subproblems.
\begin{thm} \label{thm:stat_point_lifted}
Let $(u^*, z^*, y^*)$ be any limit point of the sequence $\{(u^t, z^t, y^t)\}_{t\in\N}$ produced by Algorithm \ref{alg:multiblock_primal_dual}. Then $(u^*, z^*, y^*)$ gives rise to a critical point of problem \eqref{eq:consensus_prob_lifted}, i.e.,
\begin{align}
&0 \in \partial f(z^*)-y^*, \label{eq:opt_lifted_1} \\
&0 \in \partial g(u^*)+A^\top y^*, \label{eq:opt_lifted_2} \\
&Au^* - z^* - \lambda y^*= 0 \label{eq:opt_lifted_3},
\end{align}
with $w^*:=\lambda y^*$.
\end{thm}
%\vspace{-0.5cm}
\begin{proof}
%Let $(u^*, z^*, y^*)$ be a limit point of the sequence $\{(u^t, z^t, y^t)\}_{t\in\N}$. This means there exists a subsequence $\{t_j\}_{j\in\N} \subset \{t \}_{t\in\N}$ so that
Let $\{t_j\}_{j\in\N} \subset \{t \}_{t\in\N}$ be the subindices such that
$\lim_{j\to \infty} (u^{t_j}, z^{t_j}, y^{t_j}) = (u^*, z^*, y^*)$.
The optimality condition for the $u$-update is given as
\begin{align*}
0 \in &\, \partial g(u^{t_j+1}) + \frac{1}{\sigma}(u^{t_j+1}- u^{t_j}) + A^\top y^{t_j} \\
&+ \rho A^\top (A u^{t_j} -z^{t_j} - \lambda y^{t_j}).
\end{align*}
The optimality condition for the $z$-update is given as
$$
0 \in \partial f(z^{t_j+1}) - y^{t_j} -\rho (A u^{t_j+1} -z^{t_j+1} - \lambda y^{t_j}).
$$
The $y$-update gives
$$
y^{t_j+1}= \tfrac{1}{1 + \rho \lambda}(y^{t_j} + \rho(A u^{t_j+1} - z^{t_j+1})).
$$
Letting $j\to \infty$, by Lemma \ref{lem:limit_feasibility} and the closedness of $\partial f$ and $\partial g$, we obtain \eqref{eq:opt_lifted_1}, \eqref{eq:opt_lifted_2} and \eqref{eq:opt_lifted_3}.
\end{proof}
We conclude this section with the following remark.
\begin{remark}
The above result reveals that we obtain a solution to the unregularized problem \eqref{eq:consensus_prob}, where $e_\lambda f$ is replaced with $f$, up to a violation of the linear constraint absorbed by $\lambda y^*$; see Eq.~\eqref{eq:opt_lifted_3}. For $f$ Lipschitz continuous with modulus $L$ over $\mathrm{dom}(f)$ (which contains $z^*$), one can a-priori specify a bound 
%$\epsilon(\lambda)$ 
on the violation of the linear constraint, i.e.,
%$\|\lambda y^*\| \leq \epsilon( \lambda)$ determined by 
%\begin{align}\epsilon(\lambda):=\lambda L.\end{align} 
\begin{equation*} 
\|\lambda y^*\| \leq \lambda L.
\end{equation*}
Furthermore, this also implies that in the limit case $\lambda \to 0^+$ we recover the optimality conditions of the unregluarized problem \eqref{eq:consensus_prob}.
\end{remark}
\section{Optimality Qualifications} \label{sec:qualification}
In this section, we investigate when a critical point of the lifted problem translates to a critical point of the Moreau-Yosida regularized problem \eqref{eq:consensus_prob_regularized} by reversing the substitution \eqref{eq:substitution}. 
%\tb{We identify for piecewise convex functions $f$ (as a subfamily of nonconvex, nonsmooth functions) such a translation is possible under mild conditions.} 
We identify optimality qualifications, in case of \emph{prox-regularity} and \emph{piecewise convexity} respectively, under which such a translation holds true. We further show, that in the piecewise convex case, the optimality qualifications can be guaranteed to hold for a particular choice of the step size $\nicefrac{1}{\rho}$.
%This subclass of (locally) prox-regular functions includes the $\ell_0$ norm, or the ramp loss.

\subsection{Prox-Regular Functions}
It is easy to see that, for convex $f$, reversing the substitution yields a critical point. Here we show that this assumption can be relaxed to $f$ being prox-regular at a limit point of the iterates. A prox-regular function behaves locally like a (semi-)convex function in the sense that its Moreau envelope (with sufficiently large $\lambda$) is locally Lipschitz differentiable and the associated proximal mapping is single-valued in a small neighborhood of the input argument.
We first define prox-bounded and prox-regular functions according to \cite{Rockafellar-Variational-Analysis}.
\begin{definition}[prox-boundedness]
A function $f:\R^{m} \to \R \cup \{ \infty \}$ is prox-bounded if there exists $\lambda > 0$ such that $e_\lambda f(v) > -\infty$ for some $v \in \R^{m}$. 
%The supremum of the set of all such $\lambda$ is the threshold $\lambda_f$ of prox-boundedness for $f$.
\end{definition}
\begin{definition}[prox-regularity] A function $f:\R^{m} \to \R \cup \{ \infty \}$ is prox-regular at $\bar{z}$ for $\bar{y}$ if $f$ is finite and locally lsc at $\bar{z}$ with $\bar{y}\in \partial f(\bar{z})$, and there exist $\epsilon > 0$ and $r \geq 0$ such that
\begin{align*}
f(z')\geq f(z) + \langle y, z' - z\rangle -\frac{r}{2}\|z'-z\|^2
\end{align*}
for all $\|z' -\bar{z}\|<\epsilon$,
when $y\in \partial f(z)$, $\|y-\bar{y}\| <\epsilon$, $\|z-\bar{z}\| <\epsilon$, $f(z)<f(\bar{z})+\epsilon$. 
%When this holds for all $\bar{y}\in \partial f(\bar{z})$, f is said to be prox-regular at $\bar{z}$.
\end{definition}

\begin{thm} \label{thm:stat_point_regularized_prox_reg}
Let $(u^*, z^*, y^*)$ be any limit point of the sequence $\left\{(u^t, z^t, y^t)\right\}_{t\in\N}$ produced by Algorithm \ref{alg:multiblock_primal_dual}. Let $f$ be prox-regular at $z^*$ for $y^*\in\partial f(z^*)$, and also prox-bounded. 
Set $v^*:=z^* + \lambda y^*$. Then, for $\lambda>0$ sufficiently small, $(u^*,v^*, y^*)$ corresponds to a critical point of the regularized problem \eqref{eq:consensus_prob_regularized}:
\begin{align}
&0 = \nabla e_{\lambda}f(v^*) - y^*, \label{eq:opt_regularized_prox_reg}\\
&0 \in \partial g(u^*) + A^\top y^*, \label{eq:opt_regularized_2} \\
&A u^* - v^* = 0 \label{eq:opt_regularized_3}.
\end{align}
\end{thm}
\begin{proof}
Conditions \eqref{eq:opt_regularized_2} and \eqref{eq:opt_regularized_3} follow directly from Theorem \ref{thm:stat_point_lifted}.
From Theorem \ref{thm:stat_point_lifted} we also know $0 \in \partial f(z^*) - y^*$, or equivalently $0 \in \partial (f(\cdot) -\langle \cdot, y^* \rangle)(z^*)$.
Since $f(\cdot) - \langle \cdot, y^* \rangle$ is prox-regular at $z^*$ with subgradient $0$ %(cf.~\cite[Exercise 13.35]{Rockafellar-Variational-Analysis}) 
and also prox-bounded, we can apply \cite[Proposition 13.37]{Rockafellar-Variational-Analysis} and obtain that $\nabla e_{\lambda} (f (\cdot)- \langle \cdot, y^* \rangle)$ is differentiable at $z^*$ with $0 = \nabla e_{\lambda} (f(\cdot) - \langle \cdot, y^* \rangle)(z^*)$, for some $\lambda > 0$ sufficiently small.
A straightforward calculation shows that $
e_{\lambda} (f(\cdot) - \langle \cdot, y^* \rangle)(z) =e_{\lambda}f(z + \lambda y^*) - \tfrac{\lambda}{2} \|y^*\|^2 -\langle z, y^* \rangle$.
Further differentiating both sides at $z=z^*$ yields \eqref{eq:opt_regularized_prox_reg} with $v^*=z^* + \lambda y^*$. %, which concludes the proof. 
\end{proof}
As prox-regularity is a local property \cite{Rockafellar-Variational-Analysis}, it serves as a useful tool to prove local convergence results; see, e.g., \cite{ochs2016local}. However, in our case, we are rather interested in global (subsequential) convergence. In this sense the above theorem comes with a caveat.
\begin{remark}
There may be a cyclic dependency between the step size $\nicefrac{1}{\rho} < \lambda$ of the algorithm and the choice of $\lambda$. A smaller $\lambda$ requires the choice of the step size $\nicefrac{1}{\rho}$ to be smaller, which may alter the limit point. This makes it impossible to a-priori specify a feasible pair of parameter $\lambda$ and step size $\nicefrac{1}{\rho}$, as long as there exists no uniform $\lambda$ such that $f$ is prox-regular everywhere. 
\end{remark}
%In the next section we study a special case, piecewise convex functions with finitely many pieces, that are prox-regular (even convex) almost everywhere. This allows us to devise a more explicit optimality qualification.
As an exception to the above remark, a \emph{semi-convex} function $f$ is prox-regular everywhere with uniform $\lambda$ as in, e.g., \cite{ArtinaFS13,moellenhoff-siims-15}. Nonetheless, the semi-convexity assumption is often too strong in practice, e.g.,~for piecewise convex functions studied in the next subsection.

\subsection{Piecewise Convex Functions} \label{sec:piecewiseconvex}
In the remainder of this section we consider $f:\R^m \to\R \cup \{\infty\}$ being a piecewise convex function with finitely many pieces.
Such a function can be expressed in terms of a pointwise minimum over convex functions (which include indicator functions):
\begin{align*}
f(z):= \min_{i\in \mathcal{I}} f_i(z),
\end{align*}
with each $f_i:\R^m \to\R \cup \{\infty\}$ proper, convex, lsc and a finite index set $\mathcal{I}$. 
Its Moreau envelope $e_\lambda f$ is given in terms of a pointwise minimum over the Moreau envelopes of the individual functions $\{f_i\}$, i.e.,
\begin{align*} 
e_\lambda f(z) = \min_{i\in \mathcal{I}} e_\lambda f_i(z),
\end{align*}
and therefore remains nonsmooth. Due to \cite[Theorem 2.26]{Rockafellar-Variational-Analysis}, each individual piece $e_\lambda f_i$ is continuously differentiable, as $f_i$ is convex, proper, and lsc.
Practically, this class is relevant in statistics and machine learning. 
For examples we refer to Sections \ref{sec:intro} and \ref{sec:numerics}.
Even though $f$ is prox-regular almost everywhere, Theorem \ref{thm:stat_point_regularized_prox_reg} cannot be applied conveniently due to the fact that a qualified $\lambda$ rather depends on the (a-priori unknown) limit point. Hence, our goal is to prove a more explicit optimality qualification for piecewise convex $f$ and a presumed $\lambda$, 
so as to overcome the limitation of Theorem \ref{thm:stat_point_regularized_prox_reg}. 
To this end, we first define the set of active indices at $z$ as follows. 

\begin{definition}[active set]
Let $z \in \R^{m}$ with $f(z)$ finite. The active index set $\mathcal{A}_f(z)$ of $f$ at $z$ is defined as
\begin{align*}
\mathcal{A}_f(z) = \left\{i \in \mathcal{I} : f(z) = f_{i}(z) \right\}.
\end{align*}
\end{definition}
%It is easy to see that for convex $f_i$, $f$ is prox-regular at points $z$, where $\mathcal{A}_f(z)$ is a singleton. Conversely, points $z$, where the cardinality of $\mathcal{A}_f(z)$ is greater than one, form a null set since $\mathcal{I}$ is finite.
%Even though Theorem \ref{thm:stat_point_regularized_prox_reg} applies, since $f$ is prox-regular almost everywhere, the condition is rather artificial, not least because of the cyclic dependency of $\rho$ and $\lambda$ pointed out above.
We will show that the following qualification condition 
\begin{align} \label{eq:qualification_condition}
\mathcal{A}_f(z^*)\subset \mathcal{A}_{e_\lambda f}(z^*+\lambda y^*)
\end{align}
guarantees optimality of the limit points.
In essence, the condition requests that, after translation from $z^*$ to $z^*+\lambda y^*$, the same piece remains active with respect to its Moreau envelope.
%Numerically, we observe that for $t$ sufficiently large, $z^t$ enters a neighborhood $V$ of $z^t$ on which $\mathcal{A}_f(z)=\{i^*\}$ is a constant singleton for all $z\in V$. Under the assumption that the whole sequence $\{(z^t, y^t)\}_{t\in \N}$ converges (and $z^t$ stays within $V$), the condition \eqref{eq:qualification_condition} has a continuity property: 
%Since $e_\lambda f$ is continuous, $\mathcal{A}_{e_\lambda f}(p^*)$ can only grow in the limit, so that guaranteeing 
%\begin{align} \label{eq:qualification_condition_iter}
%\mathcal{A}_f(z^t) \subset \mathcal{A}_{e_\lambda f}(z^t+\lambda y^t),
%\end{align}
%for all $t \in \N$, implies \eqref{eq:qualification_condition}. 
We remark that for a particular choice of the step size $\nicefrac{1}{\rho}=\lambda$, the active-set condition \eqref{eq:qualification_condition} is achieved automatically, cf.~Theorem \ref{thm:stat_point_regularized_piecew_conv_2}. %the active-set qualification condition \eqref{eq:qualification_condition} is indeed observed numerically in our experiments.
%both along the iterates (after sufficiently many iterations) and at the limit point.

In the following, we characterize the subgradient of $f$ in terms of the (convex) subgradients of the individual pieces. In Lemma \ref{lem:subdiff_pointwise_min_incl} we show that an inclusion holds in a general setting. In Lemma \ref{lem:subdiff_pointwise_min_eq} we show that this inclusion holds with equality for the Moreau envelopes if the hypograph of $e_\lambda f$ satisfies the \emph{linear independence constraint qualification} (LICQ). 
%By assuming LICQ we exclude degenerated pieces $f_i$ that are only active at a single point.

\begin{lem} \label{lem:subdiff_pointwise_min_incl}
Let $z\in \R^m$ with $f(z)$ finite. Then the following inclusion holds:
\begin{align} \label{eq:inclusion_subdiff}
\partial f(z) \subset \bigcup_{i \in \mathcal{A}_f(z)} \partial f_i(z).
\end{align}
\end{lem}
%\vspace{-0.5cm}
\begin{proof}
Let $y \in \partial f(z)$. By definition, we have $z^t \to z$ with $f(z^t) \to f(z)$ and $y^t \in \hat{\partial} f(z^t)$ with $y^t \to y$. Since the active set is never empty we may choose $i^t \in \mathcal{A}_f(z^t)$ for each $t$. Since $\{i^t\}_{t\in\N}$ is discrete and $f(z^t) \to f(z)$, there exists a subsequence $\{t_j\}_{j\in\N} \subset \{t\}_{t\in\N}$ such that $i^{t_j} = i^*$ for all $j\in\N$ with some constant $i^* \in \mathcal{A}_f(z)$.
Since $y^{t_j} \in \hat{\partial} f(z^{t_j})$ and $f_{i^*}(z^{t_j}) = f(z^{t_j})$ and since $f_{i^*}(\cdot) \geq f(\cdot)$, we have
\begin{align*}
0&\leq \liminf_{z' \to z^{t_j}} \frac{f(z')-f(z^{t_j})- \langle y^{t_j},z'-z^{t_j} \rangle }{\|z'-z^{t_j}\|} \\
&\leq \liminf_{z' \to z^{t_j}} \frac{f_{i^*}(z')-f_{i^*}(z^{t_j})- \langle y^{t_j},z'-z^{t_j} \rangle }{\|z'-z^{t_j}\|},
\end{align*}
which implies that $y^{t_j} \in \hat{\partial} f_{i^*}(z^{t_j})$ for all $j$. Passing $j\to\infty$, we conclude that $y \in \partial f_{i^*}(z)$.
\end{proof}

\begin{lem} \label{lem:subdiff_pointwise_min_eq}
Let $v\in \R^m$ with $e_\lambda f(v)$ finite. Assume that the set of active normals at $v$, which is defined by
%vectors $(\nabla e_\lambda f_i(p), -1)^\top$
\begin{align} \notag
%\label{eq:linear_ind_normals}
\U_v:=\left\{ (\nabla e_\lambda f_i(v), -1)^\top: i \in \mathcal{A}_{e_\lambda f}(v)  \right\},
\end{align} 
%of the active pieces $i \in \mathcal{A}_{e_\lambda f}(p)$ 
is linearly independent. Then \eqref{eq:inclusion_subdiff} holds with equality for the Moreau envelope $e_\lambda f$:
\begin{align} \notag
%\label{eq:equality_subdiff}
\partial e_\lambda f(v) = \left\{ \nabla e_\lambda f_i(v): i \in \mathcal{A}_{e_\lambda f}(v) \right\}.
\end{align}
\end{lem}
\begin{proof}
%Let $p\in \R^m$. 
%We denote by $u_i:=(\nabla e_\lambda f_i(p), -1)^\top$, $i \in \mathcal{A}_{e_\lambda f}(p)$. 
Note that, for each $i\in \mathcal{A}_{e_\lambda f}(v)$, $e_\lambda f_i$ is continuously differentiable since $f_i$ is proper, convex, and lsc (cf.~\cite[Theorem 2.26]{Rockafellar-Variational-Analysis}). To show the desired result we construct for any $i \in \mathcal{A}_{e_\lambda f}(v)$ a direction $d_i \in \R^m$ such that for $\tau >0$ sufficiently small, $i$ is the only active piece at $v+ \tau d_i$, i.e.,~$\mathcal{A}_{e_\lambda f}(v + \tau d_i)=\{i\}$. Then clearly $e_\lambda f$ is differentiable at $v + \tau d_i$ with $\nabla e_\lambda f(v + \tau d_i)=\nabla e_\lambda f_i(v + \tau d_i) $, which proves that $\nabla e_\lambda f_i(v) \in \partial e_\lambda f(v)$. 

To this end, we define 
%$l:=\mathrm{card}\,\mathcal{A}_{e_\lambda f}(p)$ denote the cardinality of $\mathcal{A}_{e_\lambda f}(p)$ 
$l:=|\mathcal{A}_{e_\lambda f}(v)|$
and the matrix $U \in \R^{(m+1) \times l}$ such that $U_{\cdot,i}=(\nabla e_\lambda f_i(v), -1)^\top$ for $1\leq i \leq l$.
%\begin{align}
%U:=\begin{pmatrix} \vline & &\vline \\
%	u_1 &  \hdots & u_l \\
%	 \vline & &\vline
%\end{pmatrix}.
%\end{align}
For each $i$, let $d_i\in\R^m,~\beta_i\in\R,~\alpha_i\in\R^l$ such that $(d_i, \beta_i)^\top=U \alpha_i$.
Now choose the $\alpha_i$ as follows. Let $\gamma_i \in \R^l$ be the $i$-th unit vector. Since by assumption $U$ has full column-rank, the following linear system $U^\top U \alpha_i = -\gamma_i$ has a unique solution $\alpha_i = (U^\top U)^{-1} \gamma_i$. This implies that $-1= \langle \nabla e_\lambda f_i(v), d_i \rangle - \beta_i < \langle \nabla e_\lambda f_j(v), d_i \rangle - \beta_i = 0$, for all $j \in \mathcal{A}_{e_\lambda f}(v)\setminus \{i\}$. For $\tau > 0$ sufficiently small this means that
\begin{align}
\langle \nabla e_\lambda f_i(v), d_i \rangle + \frac{o(\tau)}{\tau} < \langle \nabla e_\lambda f_j(v), d_i \rangle + \frac{o(\tau)}{\tau}
\end{align}
%Since $e_\lambda f_j(p)=e_\lambda f_i(p)$ for all $j \in \mathcal{A}_{e_\lambda f}(p)$ multiplying by $\tau$ and reversing the Taylor expansion yields that
 and hence
\begin{equation}
\begin{aligned}
e_\lambda f_i(v + \tau d_i) < e_\lambda f_j(v + \tau d_i),
\end{aligned}
\end{equation}
for all $j \in \mathcal{A}_{e_\lambda f}(v)\setminus \{i\}$. Thus, we verify $\mathcal{A}_{e_\lambda f}(v + \tau d_i)=\{i\}$ for $\tau\to0^+$ as desired, and the conclusion follows. 
%For any $i\in \mathcal{A}_{e_\lambda f}(p)$ we now define sequences $p_i^n:= p + \tfrac{v_i}{n} \to p$ with $e_\lambda f(p_i^n) \to e_\lambda f(p)$ and, for $n$ sufficiently big, $w_i^n := \nabla e_\lambda f_i(p_i^n) \in \hat{\partial} e_\lambda f(p_i^n)$. Since $\nabla e_\lambda f_i$ is continuous also $w_i^n \to \nabla e_\lambda f_i(p)$. This shows that $\nabla e_\lambda f_i(p) \in \partial e_\lambda f(p)$.
\end{proof}

By characterizing the hypograph of $f$, i.e., $\hyp e_\lambda f:=\{(v, q) : q \leq e_\lambda f(v) \}$, in terms of nonlinear constraints $q - e_\lambda f_i(v) \leq 0$ for all $i \in \mathcal{I}$, the assumption in Lemma \ref{lem:subdiff_pointwise_min_eq} is equivalent to the LICQ applied to $\hyp e_\lambda f$ at the point $(v, e_\lambda f(v))$.
%CQs are a concept from classical nonlinear optimization that guarantee that optimal points satisfy the KKT conditions.

%\begin{rem}
%For the remainder of this section we assume that 
%\begin{align}
%\partial f(z) = \bigcup_{i \in \mathcal{A}_f(z)} \partial f_i(z),
%\end{align}
%as well as
%\begin{align}
%\partial e_\lambda f(z) = \left\{ \nabla e_\lambda f_i(z) : i \in \mathcal{A}_{e_\lambda f}(z)  \right\}.
%\end{align}
%for piecewise convex functions, which excludes degenerate branches.
%\end{rem}

We now conclude this section with two theorems, which guarantee the stationarity of the limit points for \eqref{eq:consensus_prob_regularized} under the proposed qualification conditions. Theorem \ref{thm:stat_point_regularized_piecew_conv_1} applies to Algorithm \ref{alg:multiblock_primal_dual}, and Theorem \ref{thm:stat_point_regularized_piecew_conv_2} to Algorithm \ref{alg:proximal_penalty_method}.
\begin{thm} \label{thm:stat_point_regularized_piecew_conv_1}
Let $(u^*, z^*, y^*)$ be any limit point of the sequence $\{(u^t, z^t, y^t)\}_{t\in\N}$ produced by Algorithm \ref{alg:multiblock_primal_dual}, and $v^*:=z^*+\lambda y^*$. %for $\lambda > 0$ 
Assume %the set of active normals 
linear independence of~~$\U_{v^*}$ and the qualification condition \eqref{eq:qualification_condition}.
%\begin{align} \label{eq:qualification_condition}
%\mathcal{A}_f(z^*) \subset \mathcal{A}_{e_\lambda f}(p^*)
%\end{align}
%holds at $(u^*, z^*, y^*)$, 
Then we have
\begin{align} 
0 \in \partial e_{\lambda}f(v^*) - y^* \label{eq:opt_piecew_conv},
\end{align}
and $(u^*,v^*, y^*)$ corresponds to a critical point of the regularized problem \eqref{eq:consensus_prob_regularized}. 
%i.e.,~the conditions \eqref{eq:opt_regularized_2} and \eqref{eq:opt_regularized_3} also hold.
\end{thm}
%\vspace{-0.3cm}
\begin{proof}
Conditions \eqref{eq:opt_regularized_2} and \eqref{eq:opt_regularized_3} follow as a direct consequence of Theorem \ref{thm:stat_point_lifted}.

From Theorem \ref{thm:stat_point_lifted} we know that $0 \in \partial f(z^*) - y^*$.
From Lemma \ref{lem:subdiff_pointwise_min_incl} we know there exists $i^* \in \mathcal{A}_f(z^*)$ such that
$y^* \in \partial f_{i^*}(z^*)$. By the definition of $v^*$ we have $v^* \in z^* + \lambda \partial f_{i^*}(z^*)$. Interpreting the inclusion as the optimality condition of the proximal mapping of the convex, proper, lsc function $f_{i^*}$, we have $z^* = P_\lambda f_{i^*}(v^*)$. By \cite[Theorem 2.26]{Rockafellar-Variational-Analysis} we have that for any $\lambda > 0$ it holds that $\nicefrac{1}{\lambda}(v^* - P_\lambda f_{i^*}(v^*)) = \nabla e_\lambda f_{i^*}(v^*)$.
Rearranging the terms shows that $0 = \nabla e_\lambda  f_{i^*}(v^*) - y^*$. In view of the assumption that $i^* \in \mathcal{A}_{e_\lambda f}(v^*)$ and $\U_{v^*}$ is linearly independent, we apply Lemma \ref{lem:subdiff_pointwise_min_eq} and obtain $y^*=\nabla e_\lambda  f_{i^*}(v^*) \in \partial e_{\lambda}f(v^*)$. This proves condition \eqref{eq:opt_piecew_conv} and concludes the proof. 
\end{proof}

%The qualification condition above comprises the LICQ of the hypograph, which is satisfied everywhere in many practically relevant cases, including robust loss functions or the symmetric hinge loss.

Finally, we consider the case $\lambda \rho = 1$, i.e.,~when our method specializes to Algorithm \ref{alg:proximal_penalty_method}. We show in the following theorem that the active-set condition \eqref{eq:qualification_condition} is guaranteed by the algorithm, that solves the regularized problem \eqref{eq:consensus_prob_regularized} without further assumptions.
As Algorithm \ref{alg:proximal_penalty_method} produces no Lagrange multiplier, we set up the multiplier as $y^t:= \nicefrac{1}{\lambda} (Au^t - z^t)$.
\begin{thm} \label{thm:stat_point_regularized_piecew_conv_2}

%Let $(u^*, z^*)$ be a limit point of the sequence $\{(u^t, z^t)\}_{t\in\N}$ produced by Algorithm \ref{alg:proximal_penalty_method}. Define $v^*:=Au^*$. Then the following properties hold:
%\begin{enumerate}
%\item The qualification condition \eqref{eq:qualification_condition} is guaranteed.
%\item Define $y^*:= \nicefrac{1}{\lambda} (Au^* - z^*)$ and let $\mathcal{U}_\lambda^f(v^*)$ be linearly independent. Then $(u^*, v^*, y^*)$  is a critical point of \eqref{eq:consensus_prob_regularized}, i.e.,~the conditions \eqref{eq:opt_piecew_conv}, \eqref{eq:opt_regularized_2} and \eqref{eq:opt_regularized_3} hold.

Let $(u^*, z^*)$ be any limit point of the sequence $\{(u^t, z^t)\}_{t\in\N}$ produced by Algorithm \ref{alg:proximal_penalty_method}. 
Define $v^*:=Au^*$, $y^*:= \nicefrac{1}{\lambda} (Au^* - z^*)$.
Then the following statements hold true:
\begin{enumerate}
\item The qualification condition \eqref{eq:qualification_condition} is fulfilled.
%and $A_f(z^*)$ is a singleton.
\item Further assume that $\U_{v^*}$ is linearly independent.
%Let $y^*:= \nicefrac{1}{\lambda} (Au^* - z^*)$. 
Then $(u^*, v^*, y^*)$  is a critical point of \eqref{eq:consensus_prob_regularized}, i.e.,~conditions \eqref{eq:opt_piecew_conv}, \eqref{eq:opt_regularized_2} and \eqref{eq:opt_regularized_3} hold true.
\end{enumerate}
\end{thm}
\begin{proof}
Define $v^{t}:=Au^t$. Let $(u^*, z^*)$ be a limit point of the sequence $\{(u^t, z^t)\}_{t\in\N}$, and $\{t_j\}_{j\in\N} \subset \{t \}_{t\in\N}$ index the corresponding convergent subsequence.
Due to the $z$-update, we have $z^{t_j} \in P_\lambda f(v^{t_j})$.
Define $z^{t_j}_i:=P_\lambda f_i(v^{t_j})$ and note that $z^{t_j}=z_{i^{t_j}}^{t_j}$ for some $i^{t_j}\in \mathcal{A}_{e_\lambda f}(v^{t_j})$.
Pick an arbitrary subsequence $\{t_{j_l}\}_{l\in\N}\subset \{t_j \}_{j\in\N}$ such that $i^{t_{j_l}}=i^*$ is constant. Then $i^* \in \mathcal{A}_{e_\lambda f}(v^{t_{j_l}})$ holds for all $l$ and, due to the continuity of $e_\lambda f$, in the limit we have $i^* \in \mathcal{A}_{e_\lambda f}(v^*)$. Furthermore, $z^{t_{j_l}} = P_\lambda f_{i^*}(v^{t_{j_l}})$ holds. Since $P_\lambda f_{i^*}$ is nonexpansive, we obtain $z^* = P_\lambda f_{i^*}(v^*)$ by letting $l\to \infty$.
%Suppose $i^* \notin \mathcal{A}_f(z^*)$. Then, there exists $j \in \mathcal{I}$ so that $f_j(z^*) < f_{i^*}(z^*)$ and therefore
%\begin{align*}
%e_\lambda f(v^*) &\leq f_j(z^*) + \frac{1}{2\lambda} \|v^*-z^*\|^2 \\
%&< f_{i^*}(z^*) + \frac{1}{2\lambda} \|v^*-z^*\|^2 =e_\lambda f(v^*).
%\end{align*}
%A contradiction.

Let $i^\dagger\in\mathcal{A}_f(z^*)$, i.e., $f_{i^\dagger}(z^*) = \min_{i \in \mathcal{I}} f_{i}(z^*)$. 
Then we have
\begin{align*}
e_\lambda f(v^*) &= f_{i^*}(z^*) + \frac{1}{2\lambda} \|v^*-z^*\|^2  \\
&\geq f_{i^\dagger}(z^*) + \frac{1}{2\lambda} \|v^*-z^*\|^2 \\
&\geq \min_{i \in \mathcal{I}} \min_{z\in \R^m} f_{i}(z) + \frac{1}{2\lambda} \|v^*-z\|^2 \\
&=e_\lambda f(v^*).
\end{align*}
%We conclude that $j\in \mathcal{A}_{e_\lambda f}(v^*)$ which proves \eqref{eq:qualification_condition}. 
%%Suppose $j \neq i^*$. It holds
%%\begin{align}
%%z^* = P_\lambda f_{i^*}(v^*) = P_\lambda f_j(v^*).
%%\end{align}
%%By \cite[Theorem 2.26]{Rockafellar-Variational-Analysis} we find that
%%\begin{align}
%%\frac{1}{\lambda}(v^* - z^*) = \nabla e_\lambda f_{i^*}(v^*) = \nabla e_\lambda f_{j}(v^*),
%%\end{align}
%%with $i^*,j\in \mathcal{A}_{e_\lambda f}(v^*)$, which contradicts the linear independence of $\mathcal{U}_\lambda^f(v^*)$.
%The second part follows from a 
Thus we conclude $i^\dagger\in \mathcal{A}_{e_\lambda f}(v^*)$, which proves condition \eqref{eq:qualification_condition}.
%Now assume $\tb{i^\dagger} \neq i^*$. \tb{Then it follows that}
%\begin{align}
%z^* = P_\lambda f_{i^*}(v^*) = P_\lambda f_{\tb{i^\dagger}}(v^*).
%\end{align}
%By \cite[Theorem 2.26]{Rockafellar-Variational-Analysis} we find 
%\begin{align}
%\frac{1}{\lambda}(v^* - z^*) = \nabla e_\lambda f_{i^*}(v^*) = \nabla e_\lambda f_{\tb{i^\dagger}}(v^*),
%\end{align}
%with $i^*,\tb{i^\dagger}\in \mathcal{A}_{e_\lambda f}(v^*)$, which contradicts the linear independence of $\tb{\U_{v^*}}$. 

The second property follows from a 
similar argument as in the proof of Theorem \ref{thm:stat_point_regularized_piecew_conv_1}.
\end{proof}
%To prove the second part of the above theorem, we observe, that the proximal mapping for the $z$-update, $z^{t} \in P_\lambda f(A u^{t})$ is determined via the proximal mappings of the individual pieces $z^t_i:=P_\lambda f_i(A u^t)$, for all $i \in \mathcal{I}$ as $z^{t}=z_{i^t}^t$, for some 
%\begin{equation}
%\begin{aligned}
%i^t &\in\argmin_{i\in \mathcal{I}} \frac{1}{2\lambda} \|Au^t-z^t_i\|^2+ f_i(z^t_i)\\
%&=  \argmin_{i\in \mathcal{I}} e_\lambda f_i(Au^t).
%\end{aligned}
%\end{equation}
%This means, both, $i^t\in \mathcal{A}_{e_\lambda f}(Au^t)$ and $z^t = P_\lambda f_{i^t}(A u^t)$ hold along the iterations and due to a continuity argument also in the limit. Invoking \cite[Theorem 2.26]{Rockafellar-Variational-Analysis} we obtain the desired result. A complete formal proof is provided in the supplements.
%\begin{align}
%\frac{1}{\lambda}(A u^* - z^*) = \nabla e_\lambda f_{i^*}(A u^*).
%\end{align}

\section{Numerical Experiments} \label{sec:numerics}
In this section, we evaluate the performances of our algorithms on optimizing the regularized model \eqref{eq:consensus_prob_regularized}, in comparison with existing methods, namely linearized ADMM \cite{chambolle2011first,lin2011linearized}, vanilla ADMM \cite{eckstein1992douglas} and PALM \cite{bolte2014proximal}. 
In particular, for both vanilla and linearized ADMM we implemented prox-updates $P_{\nicefrac{1}{\rho}}e_\lambda f$ on $e_\lambda f$, and for our algorithms and PALM we rather implemented prox-updates $P_{\nicefrac{1}{\rho}}f$ on $f$.
We show that our methods consistently behave favorably in terms of lower objective value and vanishing optimality gap, in the tasks of robust linear regression and joint feature selection and semi-supervised learning with linear classifiers. In view of \eqref{eq:opt_piecew_conv}, \eqref{eq:opt_regularized_2} and \eqref{eq:opt_regularized_3}, the optimality gap is defined as
\begin{equation} \notag
\begin{aligned}
\mathrm{gap} &:= \mathrm{dist}(0, \partial e_\lambda f(v^*)-y^*) \\
&\qquad+ \mathrm{dist}(0, \partial g(u^*)+A^\top y^*) + \|Au^* - v^*\|.
\end{aligned}
\end{equation}
Due to Lemma \ref{lem:subdiff_pointwise_min_eq} and the relation between $\nabla e_\lambda f_i$ and $P_\lambda f_i$ \cite[Theorem 2.26]{Rockafellar-Variational-Analysis}, computing an optimality gap is convenient.

\begin{figure}[!htb]
\centering
\begin{subfigure}[b]{0.49\linewidth}
        \centering
        \includegraphics[width=\textwidth]{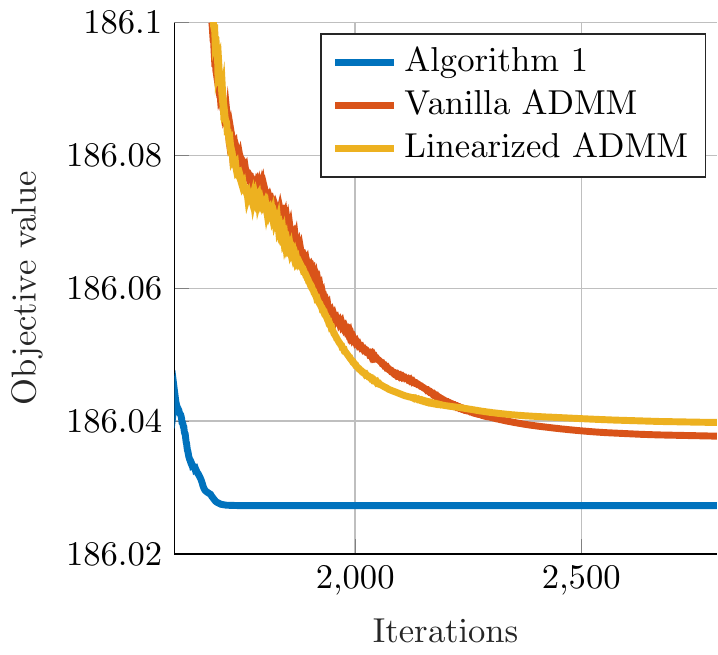}
%        \resizebox{\textwidth}{!}{
%	\input{./alg1_vs_admm_rob_regr_obj.tex}
%	}
\end{subfigure}
%arXiv \vspace{1cm}
\begin{subfigure}[b]{0.49\linewidth}
        \centering
        \includegraphics[width=\textwidth]{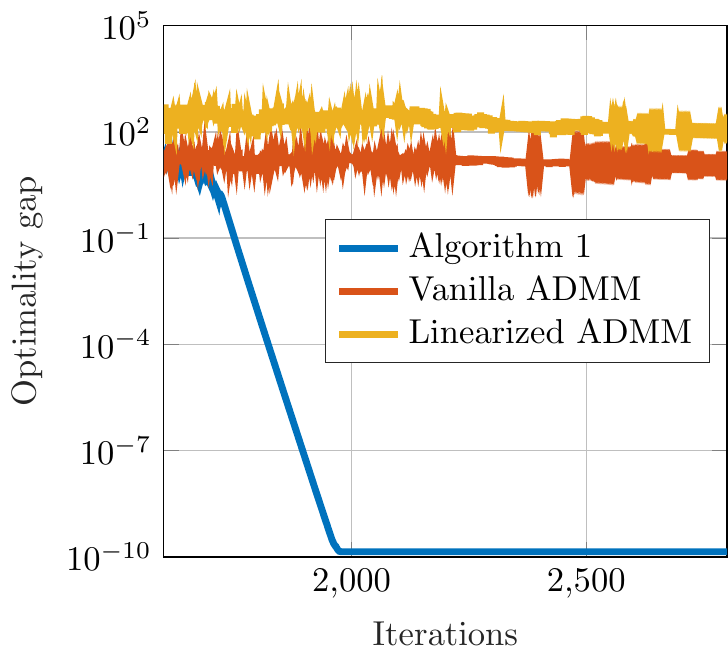}
%        \resizebox{\textwidth}{!}{
%	\input{./alg1_vs_admm_rob_regr_gap.tex}
%	}
\end{subfigure} \\
\begin{subfigure}[b]{0.49\linewidth}
        \centering
        \includegraphics[width=\textwidth]{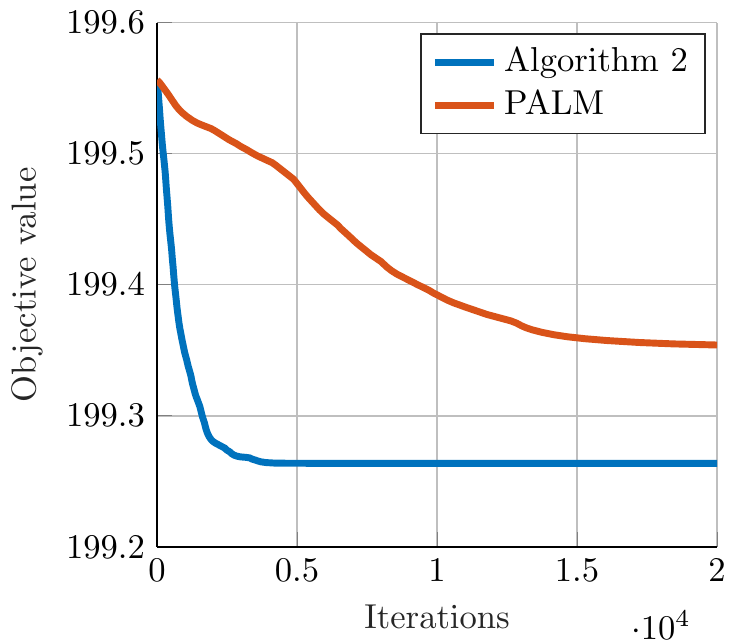}
%        \resizebox{\textwidth}{!}{
%	\input{./alg2_vs_palm_rob_regr_obj.tex}
%	}
\end{subfigure}
%arXiv \vspace{1cm}
\begin{subfigure}[b]{0.49\linewidth}
        \centering
        \includegraphics[width=\textwidth]{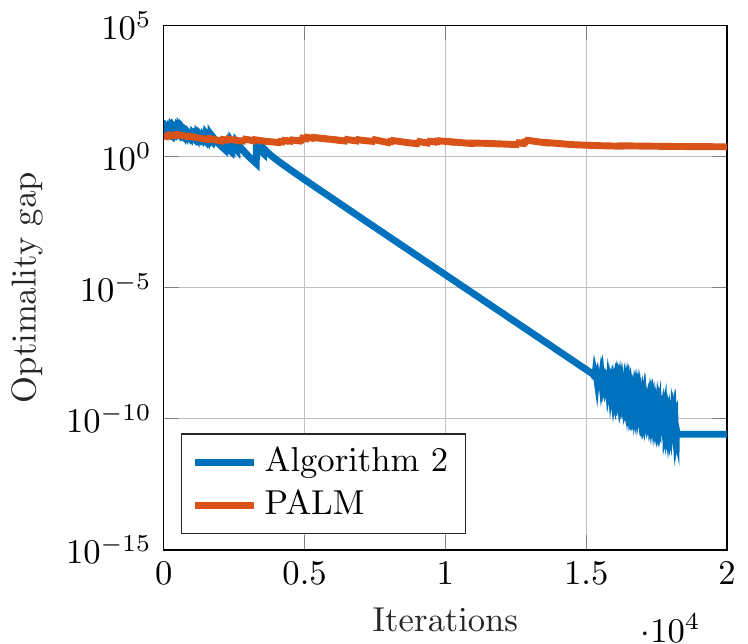}
%        \resizebox{\textwidth}{!}{
%	\input{./alg2_vs_palm_rob_regr_gap.tex}
%	}
\end{subfigure}
%\caption{Performance comparison on robust regression task. Upper row: Comparison of primal-dual type methods (linearized ADMM \cite{chambolle2011first}, vanilla ADMM \cite{eckstein1992douglas} and Algorithm \ref{alg:multiblock_primal_dual}). Lower row: Comparison of primal methods (PALM \cite{bolte2014proximal} and Algorithm \ref{alg:proximal_penalty_method}). In can be seen that our Algorithms converge to critical points of \eqref{eq:consensus_prob_regularized} with low objective value. Both ADMM and PALM fail to converge to a critical point of the regularized problem.}
\caption{Experimental comparison on robust regression. Upper row: Comparison of primal-dual type methods (linearized ADMM \cite{chambolle2011first,lin2011linearized}, vanilla ADMM \cite{eckstein1992douglas} and Algorithm \ref{alg:multiblock_primal_dual}). Lower row: Comparison of primal methods (PALM \cite{bolte2014proximal} and Algorithm \ref{alg:proximal_penalty_method}). It is observed that our algorithms converge to critical points of \eqref{eq:consensus_prob_regularized} with low objective value. Both ADMM and PALM fail to converge to a critical point of \eqref{eq:consensus_prob_regularized}.}
\label{fig:robust_regression}
%\vspace{-0.3cm}
\end{figure}
%Particularly, for our examples, the considered objectives are separable and each summand is 1d and simple. For the possibly set valued conditions \eqref{eq:opt_piecew_conv}, \eqref{eq:opt_regularized_2} and \eqref{eq:opt_regularized_3} we sum the squared euclidean distances.

%Further, we verify that the devised qualification conditions are always met at the limit points across the different scenarios. 
%As a convenient stopping criterion, we use the difference of consecutive iterates.
%\tb{In all test cases, we terminate an iterative algorithm if either the relative difference of consecutive iterates drops below a threshold or the maximum iteration number is reached.}
{\renewcommand\arraystretch{0.8}
\renewcommand\tabcolsep{4.3pt}
\begin{table*}[]
\caption{Comparison of linearized ADMM \cite{chambolle2011first,lin2011linearized}, and Algorithm \ref{alg:multiblock_primal_dual} on joint feature selection and semi-supervised learning. The results are consistent with the previous experiment. Note that for the supervised case, due to the convexity and smoothness of $e_\lambda f$, linearized ADMM and our algorithm converge to the same solution.}
%\vspace{-0.3cm}
\center
{\scriptsize
\begin{tabular}{ccccccccccc}
\toprule
& \multicolumn{5}{c}{linearized ADMM} & \multicolumn{5}{c}{Algorithm \ref{alg:multiblock_primal_dual}} \\
\cmidrule(lr){2-6} \cmidrule(lr){7-11} 
$\nicefrac{l}{N}$ & Test Error & Objective & $\nicefrac{\|u\|_0}{d}$& Iterations & Gap & Test Error & Objective & $\nicefrac{\|u\|_0}{d}$ & Iterations & Gap \\
\midrule
0.10 \%  & 1.00 \% & 138.62 & 1.00 \% & 60000 & 12 & 0.95 \% & $\mathbf{107.38}$ & 0.60 \% & 54518 & $\mathbf{1.2\cdot 10 ^{-5}}$ \\[0.5ex]
0.50 \%  & 1.00 \% & 138.29 & 0.40 \% & 60000 & 10 & 0.97 \% & $\mathbf{109.27}$ & 0.40 \% & 49880 & $\mathbf{1.1\cdot 10 ^{-5}}$  \\[0.5ex]
5.00 \%  & 1.05 \% & 132.57 & 2.59 \% & 60000 & 10 & 1.02 \% & $\mathbf{105.80}$ & 2.79 \% & 50313 & $\mathbf{1.3\cdot 10 ^{-5}}$  \\[0.5ex]
7.00 \%  & 1.10 \% & 127.23 & 3.59 \% & 60000 & 9.5 & 1.12 \% & $\mathbf{101.10}$ & 2.99 \% & 54876 & $\mathbf{1.2\cdot 10^{-5}}$  \\[0.5ex]
9.00 \%  & 1.05 \% & 127.52 & 4.38 \% & 60000 & 9.3 & 1.12 \% & $\mathbf{101.14}$ & 3.98 \% & 53439 & $\mathbf{1.2\cdot 10 ^{-5}}$  \\[0.5ex]
10.00 \%  & 1.12 \% & 128.95 & 4.38 \% & 60000 & 8.9 & 1.07 \% & $\mathbf{103.72}$ & 3.98 \% & 49452 & $\mathbf{1.1\cdot 10 ^{-5}}$  \\[0.5ex]
100.00 \%  & 1.40 \% & $\mathbf{116.34}$ & 47.61 \% & 56815 & $\mathbf{1.5 \cdot 10 ^{-5}}$ & 1.40 \% & $\mathbf{116.34}$ & 47.61 \% & 56815 & $\mathbf{1.5\cdot 10 ^{-5}}$\\[0.5ex]
\bottomrule
\end{tabular}
}
\label{tab:semi_supervised_learning}
%\vspace{-0.3cm}
\end{table*}}

\subsection{Robust Linear Regression}
In linear regression one is interested in reconstructing a signal $u \in \R^n$ from noisy measurements $b \in \R^m$. The forward model takes the form of
\begin{align*}
b=Au + \epsilon,
\end{align*}
where $A \in \R^{m \times n}$ describes the linear sampling
and
%$\epsilon \in \R^m$ is a disturbance term and 
%$A \in \R^{m \times n}$ describes the linear model.
$\epsilon\in\R^m$ is a disturbance term.

Instead of plain least squares, we use the truncated quadratic loss $e_\lambda f_i(v_i)=\min\,\{\nu, \frac{1}{2\lambda}v_i^2\}$, which is more robust against outliers \cite{she2011outlier,LiuMinimizingTQ}. We set $g(u)=0$ and the data term $e_\lambda f(v)=\sum_{i=1}^n e_\lambda f_i(v_i)$ with $f_i(v):= \nu \|v_i\|_0$ chosen as the $\ell_0$-norm.

We benchmark linearized ADMM \cite{chambolle2011first,lin2011linearized}, vanilla ADMM \cite{eckstein1992douglas}, PALM \cite{bolte2014proximal} and Algorithms \ref{alg:multiblock_primal_dual} and \ref{alg:proximal_penalty_method} on synthetic data. The entries of $A\in \R^{20000 \times 10}$ and $b\in \R^{20000}$ are i.i.d.~and normal distributed. We further degrade $b$ with additive Gaussian and high impulsive noise by adding a large constant to $60\%$ of the entries. We manually choose $\lambda=0.05$ and $\nu=0.01$. For both ADMM and our proposed Algorithm \ref{alg:multiblock_primal_dual}, a warmup phase is launched, where the parameter $\rho$ is initialized with a small value and then grows exponentially along iterations up to a value slightly bigger than $\nicefrac{1}{\lambda}=20$. In practice, such configuration often leads to lower objective values. To enforce convergence of vanilla and linearized ADMM, we keep increasing $\rho$ until $\rho=8000\gg20$.
Yet, as shown in Figure \ref{fig:robust_regression}, both ADMM and PALM fail to converge to a critical point of \eqref{eq:consensus_prob_regularized} evidenced by a non-vanishing optimality gap. In contrast, our Algorithm finds critical points with lower objective values.

\subsection{Joint Semi-supervised Learning and Feature Selection}
To further demonstrate versatility of our model, we consider the problem of joint feature selection \cite{bi2003dimensionality} and semi-supervised learning with linear classifiers. We use a Huberized semi-supervised SVM model \cite{bennett1999semi,joachims1999transductive,collobert2006large} and a nonconvex, sparsity-promoting regularizer on the classifier. The overall task is to learn a classifier from both labeled and unlabeled examples and, simultaneously, select the features.

We set up the individual components of model \eqref{eq:consensus_prob_regularized} as follows. Let $N$ be the number of training examples, among which $l\leq N$ examples are labeled and the rest is unlabeled. Let $X \in \R^{N \times d}$ be the feature matrix and $u\in \R^d$ the linear classifier. We introduce a linear constraint $X u = v$. We model the term $e_\lambda f$ as the sum of two terms: the first corresponds to the $N-l$ unlabeled training examples; the second term corresponds to the the labeled examples. More explicitly, the data term reads
\begin{align} \notag
e_\lambda f(v)=  \sum_{i=1}^{N-l} e_\lambda f_i(v_i) + \sum_{j=N-l+1}^{N} e_\lambda f_j(v_j).
%\vspace{-0.3cm}
\end{align}
For $i \leq N-l$ and $f_i(v_i)=\min_{\theta_i \in\{-1,1\}} (1-v_i\theta_i)_+$, each summand $e_\lambda f_i(v_i)$ is a ``symmetric Huberized hinge loss'' term, whose shape is depicted in Figure~\ref{fig:moreau_envelope}. For $j > N-l$, the label $\theta_j \in\{-1,1\}$ is fixed and $e_\lambda f_j(v_j)$ is the plain Huberized hinge loss for $f_j(v_j)=(1-v_j\theta_j)_+$.

For feature selection, we promote sparsity on the classifier $u$ by the $\ell_0$-norm regularization. We also include a squared $\ell_2$ norm, $\beta \|u\|^2$, to control the margin in the SVM model and ensure the coercivity of the model. Altogether, the regularization term $g$ is set up as $g(u)=\alpha \|u\|_{0}+ \beta \|u\|^2$.
To avoid extremal solutions, we fix the bias in the SVM model to the empirical mean of the data, which amounts to solving the $\mathit{NP}$-hard Furthest Hyperplane Problem (FHP) \cite{karnin2012unsupervised}.

In Table \ref{tab:semi_supervised_learning} we benchmark linearized ADMM \cite{chambolle2011first,lin2011linearized} vs. Algorithm \ref{alg:multiblock_primal_dual} on synthetic data ($N=12000$ examples) that is not linearly separable and degraded by $d=500$ additional feature components containing noise. We manually choose $\alpha=0.025$, $\beta=0.416$, $\lambda=0.5$. We stop the algorithm, when the difference of consecutive iterates is below a threshold $\epsilon=10^{-7}$ or the maximum number of $60000$ iterations is reached. The results are consistent with the previous example: Linearized ADMM does not converge to a critical point within the maximum number of iterations, except for the supervised case $l=N$. In this case, Algorithm \ref{alg:multiblock_primal_dual} and linearized ADMM converge to the same result, due to the convexity of $f$ resp. smoothness of $e_\lambda f$. 

\section{Conclusion}
In this work we have tackled highly nonconvex Moreau-Yosida regularized composite problems, where both terms in the objective are nonsmooth. Classical proximal splitting algorithms such as nonconvex ADMM fail to converge for this problem class. To overcome this limitation, we devised a novel primal-dual proximal splitting algorithm that intrinsically regularizes the behavior of the dual variable.
%via the introduction of an additional primal block realizing the Moreau-Yosida regularization. 
For piecewise convex functions, we derived explicit qualification conditions that guarantee convergence to a critical point of the Moreau-Yosida regularized problem. We validated our method on the optimization of challenging highly nonconvex machine learning objectives. For future work we will address a randomized variant of our algorithm suited to distributed computation in large-scale machine learning.

\subsubsection*{Acknowledgements} 
We would like to thank Matthias Vestner and Thomas M\"ollenhoff for fruitful discussions and helpful comments.
We gratefully acknowledge the support of the ERC Consolidator Grant 3D Reloaded.

{\small
\bibliographystyle{plain}
\bibliography{references}
}

\newpage
\appendix
\section{Proofs}
\subsection{Proof of Lemma \ref{lem:boundedness_suff_desc}}
\begin{proof}
(Statements 1 \& 2)
To show the lower boundedness of $\mathfrak{Q}_{\rho}(u^{t+1}, z^{t+1}, y^{t+1})$ we rewrite
\begin{align*}
&\mathfrak{Q}_{\rho}(u^{t+1}, z^{t+1}, y^{t+1})=
%f(z^{t+1}) +g(u^{t+1})  \\
%&\qquad+ \frac{\rho}{2} \| A u^{t+1} - z^{t+1}- \lambda y^{t+1}\|^2 \\
%&\qquad+ \frac{1}{2\lambda}\|A u^{t+1} - z^{t+1}\|^2 - \frac{\lambda}{2}\|y^{t+1}\|^2 \\
%&\qquad+ \langle A u^{t+1} - z^{t+1} , y^{t+1} \rangle \\
%&\qquad- \frac{1}{2\lambda}\|A u^{t+1} - z^{t+1}\|^2\\
f(z^{t+1}) +g(u^{t+1})  \\
&\qquad+ \frac{\rho}{2} \| A u^{t+1} - z^{t+1}- \lambda y^{t+1}\|^2 \\
&\qquad+ \frac{1}{2\lambda}\|A u^{t+1} - z^{t+1}\|^2 \\
&\qquad- \frac{1}{2\lambda} \| A u^{t+1} - z^{t+1}- \lambda y^{t+1}\|^2.
\end{align*}

Since $\rho > \frac{1}{\lambda}$ we can further bound $\mathfrak{Q}_{\rho}(u^{t+1}, z^{t+1}, y^{t+1})$ from below by the quadratic penalty $Q(u^{t+1}, z^{t+1})$ in \eqref{eq:penalty_function}:
\begin{align*}
&\mathfrak{Q}_{\rho}(u^{t+1}, z^{t+1}, y^{t+1})\geq Q(u^{t+1}, z^{t+1}).
\end{align*}
We further bound $Q(u^{t+1}, z^{t+1})$:
\begin{align*}
Q(u^{t+1}, z^{t+1})&\geq e_{\lambda} f(A u^{t+1}) + g(u^{t+1}),
\end{align*}
which is bounded from below. \\
(Statement 3)
%To show the sufficient decrease of $\mathfrak{Q}_{\rho}(u^{t+1}, z^{t+1}, y^{t+1})$ we split the difference above at the iterations $t+1$ and $t$ into the terms 
%%\begin{equation*} 
%\begin{align}
%&\mathfrak{Q}_{\rho}(u^{t+1}, z^{t+1}, y^{t+1})- \mathfrak{Q}_{\rho}(u^{t}, z^{t}, y^{t}) \\
%&=\mathfrak{Q}_{\rho}(u^{t+1}, z^{t}, y^{t})- \mathfrak{Q}_{\rho}(u^{t}, z^{t}, y^{t}) \\
%&\qquad+\mathfrak{Q}_{\rho}(u^{t+1}, z^{t+1}, y^{t})- \mathfrak{Q}_{\rho}(u^{t+1}, z^{t}, y^{t}) \\
%&\qquad+\mathfrak{Q}_{\rho}(u^{t+1}, z^{t+1}, y^{t+1})- \mathfrak{Q}_{\rho}(u^{t+1}, z^{t+1}, y^{t}) \label{eq:qd}
%\end{align}
%%\end{equation*}
%and bound each term separately.
We find an estimate for $\mathfrak{Q}_{\rho}(u^{t+1}, z^{t}, y^{t})- \mathfrak{Q}_{\rho}(u^{t}, z^{t}, y^{t})$. 
% Let $$\mathcal{C}_{\rho}(u, z, y):=\langle A u, y\rangle+ \frac{\rho}{2} \| A u - z - \lambda y \|^2,$$ denote the differentiable part of \eqref{eq:merit}, so that
%\begin{align*}
%&\mathfrak{Q}_{\rho}(u^{t+1}, z^{t}, y^{t})- \mathfrak{Q}_{\rho}(u^{t}, z^{t}, y^{t}) = g(u^{t+1}) \\
%&\qquad+ \mathcal{C}_{\rho}(u^{t+1}, z^t, y^t) - g(u^t) - \mathcal{C}_{\rho}(u^t, z^t, y^t).
%\end{align*}
By the definition of $u^{t+1}$ as the global minimum of $\mathfrak{Q}_{\rho}(\cdot, z^{t}, y^{t}) + \frac{1}{2}\|(\cdot)-u^t\|_M^2$ and $M:=\frac{1}{\sigma}I - \rho A^\top A$ positive definite for $\sigma\rho\|A\|^2<1$, we have the estimate
\begin{equation*}
  \begin{aligned}
\mathfrak{Q}_{\rho}(u^{t+1}, z^{t}, y^{t}) + \frac{1}{2}\|u^{t+1} -u^t\|_M^2 \leq \mathfrak{Q}_{\rho}(u^{t}, z^{t}, y^{t}).
  \end{aligned}
\end{equation*}
We bound $\frac{1}{2}\|u^{t+1} -u^t\|_M^2$,
\begin{equation*}
  \begin{aligned}
\|u^{t+1} -u^t\|_M^2 &= \langle u^{t+1}- u^t,M (u^{t+1}-u^t )\rangle \\
&=\frac{1}{\sigma }\|u^{t+1}-u^t\|^2 - \rho \|A u^{t+1}- Au^t\|^2\\
&\geq \left(\frac{1}{\sigma} -\rho\|A\|^2 \right)\|u^{t+1}-u^t\|^2.
  \end{aligned}
\end{equation*}
This yields the estimate
\begin{equation}
  \begin{aligned}
  &\mathfrak{Q}_{\rho}(u^{t+1}, z^{t}, y^{t})- \mathfrak{Q}_{\rho}(u^{t}, z^{t}, y^{t}) \\
  &\qquad\leq \left( \frac{\rho \|A\|^2}{2} - \frac{1}{2 \sigma} \right) \|u^{t+1} - u^t\|^2, \label{eq:qdu}
  \end{aligned}
\end{equation}
which leads to a sufficient descent if $\sigma \rho \|A\|^2 < 1$.
The optimality for the $z$-update guarantees
  \begin{equation}
  \begin{aligned}
	\mathfrak{Q}_{\rho}(u^{t+1}, z^{t+1}, y^{t})- \mathfrak{Q}_{\rho}(u^{t+1}, z^{t}, y^{t}) \leq 0. \label{eq:qdz}
  \end{aligned}
\end{equation}

Finally we bound the term
  \begin{equation*}
  \begin{aligned}
&\mathfrak{Q}_{\rho}(u^{t+1}, z^{t+1}, y^{t+1})- \mathfrak{Q}_{\rho}(u^{t+1}, z^{t+1}, y^{t}) = -\frac{\lambda}{2} \|y^{t+1}\|^2 \\
&\qquad+ \frac{\lambda}{2} \|y^{t}\|^2 +\langle A u^{t+1} - z^{t+1}, y^{t+1} -y^t\rangle \\
&\qquad+\frac{\rho}{2}\|Au^{t+1} - z^{t+1} - \lambda y^{t+1}\|^2 \\
&\qquad- \frac{\rho}{2}\|Au^{t+1} - z^{t+1} - \lambda y^{t}\|^2.
  \end{aligned}
\end{equation*}
Since $\frac{1}{\rho}(y^{t+1} - y^t)+\lambda y^{t+1}=Au^{t+1} - z^{t+1}$, we can rewrite
\begin{align*}
&-\frac{\lambda}{2} \|y^{t+1}\|^2 + \frac{\lambda}{2} \|y^{t}\|^2 + \langle A u^{t+1} - z^{t+1}, y^{t+1} -y^t\rangle \\
&=-\frac{\lambda}{2} \|y^{t+1}\|^2 + \frac{\lambda}{2} \|y^{t}\|^2 + \frac{1}{\rho} \|y^{t+1} - y^t\|^2 + \lambda \|y^{t+1}\|^2\\
&\qquad  -\lambda \langle y^{t+1}, y^t \rangle \\
&=\frac{\lambda}{2} \|y^{t+1}\|^2 -\lambda \langle y^{t+1}, y^t \rangle + \frac{\lambda}{2} \|y^{t}\|^2 + \frac{1}{\rho} \|y^{t+1} - y^t\|^2 \\
&=\left( \frac{1}{\rho} +\frac{\lambda}{2} \right) \|y^{t+1} - y^t\|^2.
\end{align*}
We apply the identity $\|a+c\|^2 - \|b+c\|^2  = -\|b-a\|^2 + 2\langle a+c, a-b \rangle$ with $a:=- \lambda y^{t+1}$, $b:=- \lambda y^{t}$ and $c:=A u^{t+1}-z^{t+1}$ and obtain
\begin{align*}
&\frac{\rho}{2}\|Au^{t+1} - z^{t+1} - \lambda y^{t+1}\|^2 - \frac{\rho}{2}\|Au^{t+1} - z^{t+1} - \lambda y^{t}\|^2 \\
&= -\frac{\rho\lambda^2}{2} \|y^{t+1} - y^t\|^2 \\
&\qquad- \lambda\rho \langle A u^{t+1}-z^{t+1} - \lambda y^{t+1}, y^{t+1} -y^t\rangle \\
&= -\frac{\rho\lambda^2 + 2 \lambda}{2} \|y^{t+1} - y^t\|^2.
\end{align*}
Overall we have:
 \begin{equation}
  \begin{aligned}
&\mathfrak{Q}_{\rho}(u^{t+1}, z^{t+1}, y^{t+1})- \mathfrak{Q}_{\rho}(u^{t+1}, z^{t+1}, y^{t}) \\
&\qquad= \left(\frac{1}{\rho} - \frac{\rho\lambda^2 + \lambda}{2}\right) \|y^{t+1}-y^{t}\|^2. \label{eq:qdy}
  \end{aligned}
\end{equation}
Summing \eqref{eq:qdu}--\eqref{eq:qdy}, we obtain the desired result:
\begin{equation}
\label{eq:suff_descent}
\begin{aligned}
  & \mathcal{Q}_{\rho}(u^{t+1}, z^{t+1}, y^{t+1})- \mathcal{Q}_{\rho}(u^{t}, z^{t}, y^{t}) \\
  & \quad\leq \left( \frac{\rho \|A\|^2}{2} - \frac{1}{2 \sigma} \right)\|u^{t+1}-u^t\|^2 \\ 
  &\quad \quad +\left(\frac{1}{\rho}-\frac{\rho\lambda^2 + \lambda}{2} \right) \|y^{t+1}-y^{t}\|^2.
\end{aligned}
\end{equation}
\end{proof}

\subsection{Proof of Lemma \ref{lem:limit_feasibility}}
\begin{proof}
Since $\{\mathfrak{Q}_{\rho}(u^{t}, z^{t}, y^{t})\}_{t\in\mathbb{N}}$ monotonically decreases by Lemma \ref{lem:boundedness_suff_desc}, it is bounded from above. Since $\{Q(u^t,z^t)\}_{t\in\mathbb{N}}$ is bounded from above by $\{\mathfrak{Q}_{\rho}(u^{t}, z^{t}, y^{t})\}_{t\in\mathbb{N}}$ and, furthermore, $Q$ is coercive by assumption, we assert that $\{u^t\}_{t\in\mathbb{N}}$, $\{z^t\}_{t\in\mathbb{N}}$ are uniformly bounded.

Now we sum the estimate \eqref{eq:suff_descent} from $t=1$ to $T$ and obtain due to the lower boundedness of the iterates $\mathfrak{Q}_{\rho}(u^{t}, z^{t}, y^{t})$:
\begin{align*}
- \infty &< \mathfrak{Q}_{\rho}(u^{T+1}, z^{T+1}, y^{T+1})- \mathfrak{Q}_{\rho}(u^{1},z^{1}, y^1) \\
&\leq  \left( \frac{\rho \|A\|^2 }{2} - \frac{1}{2 \sigma} \right) \sum_{t=1}^T\|u^{t+1}-u^t\|^2 \\
  &\quad +\left(\frac{1}{\rho} - \frac{\rho\lambda^2 + \lambda}{2}\right) \sum_{t=1}^T \|y^{t+1}-y^{t}\|^2.  
\end{align*}
Passing $T \to \infty$ yields that $\|u^{t+1}-u^t\| \to 0$ and $\| y^{t+1} - y^t \| \to 0$ for $\rho>\nicefrac{1}{\lambda}$ and $\sigma\rho\|A\|^2<1$.
From $\frac{1}{\rho}(y^{t+1}-y^t) =A u^{t+1} -z^{t+1} - \lambda y^{t+1}$ we have that, 
\begin{align*}
	0 &\leq \|z^t-z^{t+1}\| \\
	&= \|z^t - z^{t+1} + A(u^{t+1}-u^{t}) - A(u^{t+1}-u^{t})\\
	&\qquad +\lambda y^{t+1} - \lambda y^t - \lambda y^{t+1} + \lambda y^t \| \\
	&\leq \frac{1}{\rho} \| y^{t+1} - y^t \| + \|A\|\|u^{t+1}-u^{t}\| \\
	&\qquad +\lambda \|y^{t+1} - y^t  \| \to 0,
\end{align*}
and that $\|A u^{t} - z^t - \lambda y^t\| \to 0$. Since $\{u^t\}_{t\in\mathbb{N}}$, $\{z^t\}_{t\in\mathbb{N}}$ are uniformly bounded, also $\{y^t\}_{t\in\mathbb{N}}$ are uniformly bounded.
\end{proof}

\end{document}